%% file: ms.tex
\documentclass[reqno]{amsart}
\usepackage{color,latexsym,amsfonts,amssymb,bbm,comment,amsmath,cite, amsthm, bbold}

\usepackage{fullpage}
\usepackage{graphicx}
\usepackage{tikz}
\usepackage{dsfont}
\usepackage{bbold}

\newtheorem{theorem}{Theorem}
\newtheorem{corollary}[theorem]{Corollary}

\newtheorem{lemma}{Lemma}[section]

\theoremstyle{definition}

\newtheorem{innercustomthm}{Theorem}
\newenvironment{customthm}[1]
  {\renewcommand\theinnercustomthm{#1}\innercustomthm}
  {\endinnercustomthm}

\newcommand{\eqd}{\stackrel{\tiny d}{=}}

\DeclareFontFamily{T1}{dutchcal}{}
\DeclareFontShape{T1}{dutchcal}{m}{n}{<->s*[1.44]callig15}{}
\DeclareMathAlphabet\mathdutchcal   {T1}{dutchcal} {m} {n}

\newcommand*{\E}{\mathbb{E}}

\newcommand*{\N}{\mathbb{N}}
\renewcommand{\P}{\mathbb{P}}

\newcommand*{\R}{\mathbb{R}}

\author{Caio Alves$^1$ }
\address{$^1$  Alfr\'ed R\'enyi Institute of Mathematics, Budapest.}

\author{Rodrigo Ribeiro$^{2}$}
\address{$^2$ Departament of Mathematics, University of Colorado Boulder.
	\newline
	e-mail: {\itshape \texttt{rodrigo.ribeiro@colorado.edu}}}

\title{Spread of Infection over P.A. random graphs with edge insertion}

\date{\today \\
	$^1$ Institute of Mathematics, University of Leipzig\\
	$^2$ Departament of Mathematics, University of Colorado Boulder \\
}

\begin{document}

\begin{abstract}
In this work we investigate a bootstrap percolation process on random graphs generated by a random graph model which combines preferential attachment and edge insertion between previously existing vertices. The probabilities of adding either a new vertex or a new connection between previously added vertices are time dependent and given by a function $f$ called the edge-step function. We show that under integrability conditions over the edge-step function the graphs are highly susceptible to the spread of infections, which requires only $3$ steps to infect a positive fraction of the whole graph. To prove this result, we rely on a quantitative lower bound for the maximum degree that might be of independent interest.

\end{abstract}

\maketitle

\section{Introduction}

\input{intro_nova}

\section{Lower bound for the maximum degree}\label{sec:deglowerbound}
\input{bound_on_the_degree.tex}

\section{The outbreak phenomenon}\label{sec:bootstrap}
\input{bootstrap.tex}
\appendix
\input{appendix.tex}

{\bf Acknowledgements }  \textit{C.A.}  was supported by the Noise-Sensitivity everywhere ERC Consolidator Grant 772466. 

\bibliographystyle{plain}
\bibliography{ref}

\end{document}

%% file: intro_nova.tex



Investigation of spread of infections or propagation of information over networks is a crucial problem which arises naturally in many areas of science going from social sciences to economics. The recent pandemic and the dissemination of fake news in a highly interconnected society make evident the relevance of studies involving dissemination over networks.


An accurate model for spread of infections must accommodate somehow the randomness of the environment, since concrete situations such as human societies produce networks which evolve in time according to random rules. In this direction, random graph models whose rule of evolution combines the so-called {\it preferential attachment rule} became a natural environment to the investigation of infectious processes \cite{berger2005spread, jacob2019metastability, can2015metastability}. The reason is that this mechanism of attachment which is driven by popularity, that is, individuals tends to get connected with the more popular ones, proved that it is capable of capturing network properties shared by many networks in real-file \cite{SW98,BA99}. We do not intend to cover the vast literature of preferential attachment random graphs, but we refer the reader to the book of R. van der Hofstad \cite{vH16} and of R. Durrett \cite{DBook2} for a wide and rigorous introduction to many important random graph models.



In this work we investigate the bootstrap percolation process on random graphs generated by a random graph model which combines preferential attachment and edge insertion between previously existing vertices. In order to properly state and discuss the nature of our results we will define the underlying random graph model and latter the Bootstrap percolation process performed over our random graphs.

\subsection{The model} The model depends on a real non-negative function $f$, called the \textit{edge-step function}, with domain given by the semi-line $[1, \infty)$ such that $||f||_{\infty} \le 1$. Though the results here stated also apply for any finite given initial graph, we will assume henceforth that the process starts from an initial graph $G_1$ consisting in one vertex and one loop in order to simplify notation. As the process evolves, one of the two graph stochastic operations below may be performed on the graph~$G$: 
\begin{itemize}
    \item \textit{Vertex-step} - Add a new vertex $v$ and add an edge $\{u,v\}$ by choosing $u\in G$ with probability proportional to its degree. More formally, conditionally on $G$, the probability of attaching $v$ to $u \in G$ is given by
    \begin{equation}\label{def:PArule}
    P\left( v \rightarrow u \middle | G\right) = \frac{\mathrm{degree}(u)}{\sum_{w \in G}\mathrm{degree}(w)}.
    \end{equation}
    \item \textit{Edge-step} - Add a new edge $\{u_1,u_2\}$ by independently choosing vertices $u_1,u_2\in G$ according to the same rule described in the vertex-step. We note that both loops and parallel edges are allowed.
    
\end{itemize}

We let $\{Z_t\}_{t \in \N}$ be an independent sequence of random variables such that $Z_t\eqd \mathrm{Ber}(f(t))$. We then define a markovian random graph process $\{G_t(f)\}_{t \ge 1}$  as follows: begin with initial state~$G_1$. Given $G_{t}(f)$, obtain the (multi)graph $G_{t+1}(f)$ by either performing a \textit{vertex-step} on $G_t(f)$ when $Z_t=1$ or performing an \textit{edge-step} on $G_t(f)$ when~$Z_t=0$.

We will make use of natural numbers, mainly $i,j$ and $k$, to denote respectively the $i$-th, $j$-th and $k$-th vertex added by the random graph process. Moreover, we will let $d_t(i)$ be the {\it degree of the $i$ vertex at time~$t$}. If $i$ has not been added yet, then $d_t(i) = 0$.

\subsection{Regularity conditions} Again, in order to properly state and discuss our results we will need to introduce some regularity conditions for the functions $f$. Some of theses conditions have the objective of preventing pathological examples, for instance if one drops monotonicity it is possible to construct $f$ such that the sequence of graphs~$\{G_t(f)\}_{t \in \N}$ has two sub sequences of graphs: one similar to the BA random tree and the other one a quasi-complete graph, see the discussion in Section 8 of~\cite{ARS19b}. 

Some of our results will require information about the asymptotic behavior of $f$ and for this reason a wide class of functions will play important role: the regularly varying functions. We say that a positive function $f$ is a \textit{regularly varying function} (\textit{r.v.f} for short) at infinity with \textit{index of regular variation} $-\gamma$, for some $\gamma \geq 0$, if
\[
\lim_{t\to\infty } \frac{f(at)}{f(t)} = a^{-\gamma},
\]
for all $a > 0$. The special case $\gamma = 0$ is called \textit{slowly varying function}.
Below we define the conditions over $f$ that will be useful to our purposes. For $p\in [0,1]$, we define conditions
\begin{equation}\label{def:dp}\tag{D$_p$}
    f \text{ decreases to }p.
\end{equation}
We will also need to define some summability conditions:
\begin{equation}\tag{V$_{\infty}$}
    \sum_{s=1}^\infty f(s) =\infty.
\end{equation}
The above condition relates to the expected number of vertices. Since at each step $s$ we add a vertex with probability $f(s)$, the above condition tells us that the process keeps introducing new vertices into the graph.
\begin{equation}\label{def:condS}\tag{S}
\sum_{s=1}^{\infty} \frac{f(s)}{s} < \infty.
\end{equation}
Whenever (S) does not hold for a specific $f$, we say that $f$ satisfies (S)$^c$.
For any $\gamma \in [0,1]$, we let ${\rm RES}(-\gamma)$ be the following class of functions
\begin{equation}\tag{RES}
\begin{split}
\mathrm{RES}(-\gamma) := \left \lbrace f:[1, \infty] \longrightarrow [0,1] \; \middle | \; f  \text{ satisfies (D)}_0 \text{ and is r.v.f with index } -\gamma\right \rbrace.
\end{split}
\end{equation}

Due to degree of freedom one has to choose $f$, since it is a function, one can study this model in several different contexts.  It has been studied under different regularity conditions on $f$ and different graph properties have been investigated. In \cite{ARS17b} the authors proved that, under $\mathrm{RES}(-\gamma)$, the empirical degree distribution of~$G_t(f)$ is close to a power-law whose exponent is given by $2-\gamma$, when $\gamma \in [0,1)$. Whereas in \cite{ARS20}, the authors investigated how the diameter of $G_t(f)$ depends on~$f$ under different conditions on the asymptotic behavior of~$f$. A variety of regimes for the diameter can be achieved, from $O(1)$ regimes to $\Theta (\log t)$ ones. In \cite{ARS17}, the clique number has been investigated for $f \equiv p \in (0,1)$. 

\subsection{Main Results}
One of the central questions in the study of infectious processes is the outbreak, when the whole or a large proportion of the network becomes infected. This macro phenomena can be triggered by local decisions whose effect scales to a macro state. In a nutshell part of our results address the `amount' of local decisions needed to trigger the outbreak and the time needed to observe it. 

In order to investigate the Bootstrap percolation process we need to prove certain properties of the underlying environment. For this reason, our first result is a quantitative lower bound for the maximum degree. In order to state it, we need to introduce new notation. Given an edge-step function~$f$, we define the normalizing function~$\phi: \N \to (0, \infty)$ by
\begin{equation}
    \label{eq:phi_def}
    \phi(t) = \phi(t, f) := \prod_{s = 1}^{t - 1} \Big( 1 + \frac{1}{s} - \frac{f(s + 1)}{2s} \Big)
\end{equation}
We note that condition~\eqref{def:condS} is equivalent to~$\phi(t)$ growing linearly as a function of~$t$. We have
\begin{theorem}[Lower bound for maximum degree]
    \label{thm:deglowerbound}
    Let $f$ be an edge-step function such that $f(t)$ goes to zero as $t$ goes to infinity. For every $N \in \N$, there exists $C_f >0$ depending on $f$ only such that
    \begin{equation}
    \label{eq:deglowerbound}
        \begin{split}
            \P\left(\forall t\in\N,\exists i \in \{1,2,\cdots,N\}\text{ such that } d_t(i)\geq \frac{\phi(t)}{\phi(N)}\right) &\geq 1- \exp\{-C_f N\}.
        \end{split}
    \end{equation}
    Furthermore, there exists almost surely a random integer~$N_0\geq 0$ such that for every~$t \geq N_0$, there exists at least one vertex~$i$ in~$G_t(f)$ such that
    \begin{equation}
        \label{eq:deglowerboundmax}
        d_t(i)\geq \frac{\phi(t)}{\phi(N_0)}.
    \end{equation}
\end{theorem}

As said before, in this work we analyze the Bootstrap Percolation model over~$G_t(f)$. We will consider $f\in {\rm RES}(-\gamma)$, with $\gamma \in [0,1)$ and satisfying condition (S). Our estimates allow us to construct structures on the graph~$G_t(f)$ that make it highly susceptible to the spread of infections.

Now, let us define the Bootstrap Percolation process. Given a finite (multi)graph~$G=(V(G),E)$, a number~$a\in[0,|V(G)|]$, and an integer~$r\geq 2$, we define the \textit{bootstrap percolation} measure~$\mathbb{Q}_{G,a,r}$ on~$G$ with threshold~$r$ and rate of infection~$a$ in the following manner:
\begin{itemize}
    \item Each vertex~$v\in V(G)$ is~\emph{infected} at round~$0$ independently of the others with probability~$a|V(G)|^{-1}$. The collection of all infected vertices at round~$0$ is denoted by~$\mathcal{I}_0$. \\
    
    \item At round~$s\in \N$, every vertex connected to~$\mathcal{I}_{s-1}$ via at least~$r$ edges becomes infected. \\
    
    \item We let~$\mathcal{I}_{\infty}$ be the set of all infected vertices when the process stabilizes, that is, 
    \[
    \mathcal{I}_{\infty}=\cup_{s\geq 0} \mathcal{I}_{s}.
    \]
\end{itemize}
In~\cite{AF18} the authors study the bootstrap percolation process on the preferential attachment random graph where each vertex has~$m\in\N$ outgoing edges with end vertices chosen according to an affine preferential attachment rule, that is, the PA rule in~\eqref{def:PArule} but with a~$\delta>-m$ summed in both the numerator and denominator. There they prove the existence of a~\emph{critical function}~$a_t^c:\N\to\R_+$ such that the bootstrap percolation process on this random graph at time~$t$ with threshold~$r\leq m$ and rate~$a_t\gg a_t^c$ infects the whole graph with high probability, but the same process with rate~$a_t'\ll a_t^c$ dies out without before infecting a positive proportion of the graph, also with high probability. 

In our context it is not possible for the infection to spread to the whole graph due to the existence of many vertices with degree smaller than~$r$ at all times. We therefore turn the problem into asking if~$\mathcal{I}_\infty$ eventually encompasses a set with positive density in the vertex set. In this sense, we show in our next result that the bootstrap percolation is~$\emph{always supercritical}$ for the graph~$G_t(f)$, with $f$ under (S), i.e., every unbounded rate sequence gives rise to a set of infected vertices with positive density:
\begin{theorem}[The outbreak phenomenon]\label{thm:bootstrap}Let $f$ be an edge-step function in ${\rm RES}(-\gamma)$, with $\gamma \in [0,1)$, and satisfying the summability condition (S). Then, for any sequence $(a_t)_{t \in \mathbb{N}}$ increasing to infinity and a integer number $ r\ge 2$ there exists a collection of graphs $\mathcal{G}_t = \mathcal{G}_t((a_t)_{t \in \N},r)$ and a $(f,r)$-dependent constant~$c>0$ such that, 
    \begin{equation}
    \mathbb{P}\left( G_t(f) \in \mathcal{G}_t \right) = 1-o(1)
    \end{equation}
    and for all $G\in \mathcal{G}_t$ the bootstrap percolation process on $ G$ with parameters $a_t$ and $r\ge 2$ satisfies 
    \begin{equation}
    \mathbb{Q}_{G,a_t,r}\left( |\mathcal{I}_{\infty}| \ge c|V( G)| \right) = 1-o(1).
    \end{equation}
\end{theorem}
The above theorem illustrates how interconnected $G_t(f)$ is. From the definition of the bootstrap percolation on $G_t(f)$, it follows that the initial infected set $\mathcal{I}_0$ is essentially $a_t$. Thus, by Theorem~\ref{thm:bootstrap} one can choose $a_t$ of order arbitrarily smaller than the expected number of vertices in $G_t(f)$, as long as it increases to infinity, the infection manages to spread to a positive fraction of the whole graph with high probability. 

In order to prove Theorem \ref{thm:bootstrap} we actually prove a stronger result, which provides an upper bound for the time it takes for the infection to spread to a positive fraction of the whole graph. Formally, for a bootstrap process on $G$ with parameter $r$ and $a$, given a positive constant $c$, let $\tau_c$ be the following stopping time
\begin{equation}\label{def:tau}
    \tau_c := \inf \{ s\ge 0 \; : \; |\mathcal{I}_s| \ge c|V(G)|\}.
\end{equation}
Then we have the following result which is a stronger version of Theorem \ref{thm:bootstrap}.
\begin{customthm}{2'}[Number of steps for the outbreak] \label{thm:tau}Let be $\mathcal{G}_t$ and $c$ as in Theorem \ref{thm:bootstrap}. Then, for any sequence $(a_t)_{t \in \mathbb{N}}$ increasing to infinity, $r\ge 2$ and $G \in \mathcal{G}_t$ the following holds
    $$
    \mathbb{Q}_{G,a_t,r}\left( \tau_c \le 3 \right) = 1- o(1).
    $$
\end{customthm}
Roughly speaking, the above result states that typically, a bootstrap percolation process on $G_t(f)$ needs at most $3$ steps to reach a positive fraction of $G_t(f)$, making clearer the interconnectedness of the graphs $\{G_t(f)\}_{t\ge 1}$ when $f$ is under $(S)$ and $\mathrm{RES}(-\gamma)$.

%% file: bound_on_the_degree.tex
In this section, we will use martingale arguments in order to show a quantitative lower bound for the maximum degree of~$G_t(f)$. This type of result is usually not as strong as its upper bound counterpart, since we may pay a relatively small price in order to make a vertex~$i$ behave like it was born much after its mean appearance time. We can however obtain stronger results when dealing with the \emph{total degree of a collection of vertices}. This in turn gives us good lower bounds for the degree of some random vertex, by the pigeonhole Principle.

We start with an elementary stochastic recurrence identity concerning the degree of a given vertex~$i$. Define as~$\tau(i)$ the random time in which the~$i$-th vertex is born.

\begin{lemma}
	\label{l:degree_recurrence}
	Given the vertex with index~$i$ we have, in the event where~$s \geq \tau(i)$,
	\begin{equation}
		\label{eq:degree_increment_cond_expec}
		\begin{split}
			\mathbb{E} \left[ \Delta d_s(i)  \middle | \mathcal{F}_s\right] 
			&=
			\left( \frac{1}{s}-\frac{f(s+1)}{2s}\right)d_s(i),
		\end{split}
	\end{equation}
	which implies that for $X_{t_i, s}$ defined as
	\begin{equation}
		\label{eq:martingale degree}
		\begin{split}
			X_{t_i,s} := \frac{d_s(i)}{\phi(s)}\mathds{1}\{\tau(i)=t_i\}
		\end{split}
	\end{equation}
	the process $\{X_{t_i,s}\}_{s \ge t_i}$ is a martingale.
\end{lemma}

\begin{proof}
	We start by noticing that~$\Delta d_s(i)$ is either $0$, $1$, or~$2$. In order for it to be~$1$, we may either add a new vertex to the graph and connect it to $i$, or we add a new edge and choose $i$ to be one of its endvertices. In the case where the increment is~$2$, we add a new edge to the graph and choose $i$ twice, adding a loop to it. This yields the formula, on the event where~$s > \tau_i$,
	\begin{equation}
	\label{eq:inccondexp2}
		\begin{split}
			\mathbb{E} \left[ \Delta d_s(i)  \middle | \mathcal{F}_s\right] 
			&=
				f(s+1)\frac{d_s(i)}{2s}+2\left(1-f(s+1)\right)\frac{d_s(i)}{2s}\left( 1- \frac{d_s(i)}{2s}  \right)
				+2\left(1-f(s+1)\right) \Bigg(\frac{d_s(i)}{2s}\Bigg)^2 
			\\
			&=
				\left( \frac{1}{s}-\frac{f(s+1)}{2s}\right)d_s(i),
		\end{split}
	\end{equation}
which in turn, by the definition of~$\phi$ in \eqref{eq:phi_def}, implies that the process in~\eqref{eq:martingale degree} is a martingale.
\end{proof}

Let the function~$\xi:\N\to \R$ be defined by
	\begin{equation}
		\label{eq:xidef}
			\xi(s) = \xi(s, f) := \frac{\phi(s)}{s}=\prod_{r=1}^{s-1}\frac{r}{r+1}\prod_{r=1}^{s-1}\left( 1 + \frac{1}{r}-\frac{f(r+1)}{2r} \right)=\prod_{r=1}^{s-1}\left( 1 -\frac{f(r+1)}{2(r+1)} \right).
	\end{equation}	 
In order to properly calculate sums involving~$\phi$, the following lemma about~$\xi$ will be necessary:
\begin{lemma}
\label{lemma:xi_slowly}
If~$f$ satisfies~$(\mathrm{D}_0)$, that is, $f(t)$ decreases monotonically to~$0$ as $t$ increases, then~$\xi$ is a slowly varying function.
\end{lemma}

\begin{proof}
We may write $\xi$ as
	\begin{equation}
		\label{eq:xibound}
		\xi(s)=\exp\left\{        \sum_{r=1}^{s-1}\log\left( 1 -\frac{f(r+1)}{2(r+1)} \right)      \right\}
		= 
			\exp\left\{        \sum_{r=1}^{s-1}\left(  -\frac{f(r+1)}{2(r+1)} +O(r^{-2}) \right)    \right\}.
	\end{equation}
Therefore~$\xi$ is a slowly varying function as long as~$f(t)$ goes to zero as $t$ goes to infinity, since for any~$a\in\R_+$,
	\begin{equation}
		\label{eq:zetaslow}
		\begin{split}
			\frac{\xi(as)}{\xi(s)}
			&=
				 \exp\left\{    \sum_{r=\lceil as \rceil }^{s-1}\left(\frac{f(r+1)}{2(r+1)}  +O(r^{-2})\right) \right\}
			\\ &\leq
				 \exp\left\{ C\cdot as^{-1}   +\frac{\inf_{r \geq as}f(r)}{2}\left(\log s - \log as \right)   \right\}
			\\ &\leq
				 \exp\left\{ C\cdot as^{-1}  -  \frac{\inf_{r \geq as}f(r)}{2}\left(  \log a \right)   \right\}
			 \\
			 &\xrightarrow{s\to\infty} 1.
		 \end{split}
	\end{equation}
\end{proof}

With all the previous results at our disposal we can finally prove the main result of this section, which states a lower bound for the maximum degree of $G_t(f)$.

\begin{proof}[Proof of Theorem~\ref{thm:deglowerbound}] 
	Fix $N \in \mathbb{N}$ and denote by $\{W_{N,s}\}_{s \ge N}$ the process
	\begin{equation}
		W_{N,s} := \frac{\sum_{i=1}^N d_s(i)\mathds{1}\{\tau(i)\le N \}}{\phi(s)},
	\end{equation}
	i.e., $W_{N,s}$ denotes the sum of the degree of all vertices added by the process up to time $N$ normalized by $\phi(s)$. Since
	\begin{equation}
		\tilde W_{N, s} = \frac{1}{\phi(s)} \sum_{i \geq 1} d_s(i) \mathds{1}\{\tau(i) \leq N\},
	\end{equation}
	Lemma~\ref{l:degree_recurrence} implies that~$\{W_{N,s}\}_{s \ge N}$ is a positive martingale such that
	\begin{equation}
		\E W_{N,s} = W_{N,N} \equiv \frac{2N}{\phi(N)}.
	\end{equation}
	Indeed, at time $N$ the sum of the degree of all vertices added up to this time equals twice the number of edges, which is $N$. 
  In order to shorten the notation, we define
  \[
    D_{N,s} = \sum_{i=1}^N d_s(i)\mathds{1}\{\tau(i)\le N \}.
  \]
  Since we can update the degree of at most two vertices by an amount of at most $2$, it follows that $\{W_{N,s}\}_{s \ge N}$ has bounded increments:
	\begin{equation}
		\label{eq:XNs_bounded_increment}
		\begin{split}
			\lefteqn{
					\left|   \Delta W_{N, s}       \right|
					}\quad
			\\
			&=
				\left|  
					\frac{1}{\phi(s+1)} D_{N, s + 1}   
					-  
					\frac{1}{\phi(s)} D_{N, s}
				\right|
			\\
			&=
				\frac{1}{\phi(s+1)\phi(s)}
					\left|        
						\phi(s) D_{N, s + 1}
						-
						\left(1+\frac{1}{s}-\frac{f(s+1)}{2s}\right)\phi(s)
						D_{N, s}
					\right|
			\\
			&=
					\left|     
						\frac{1}{\phi(s+1)} \Delta D_{N, s}
						- 
						\frac{(2-f(s+1))}{2s\phi(s+1)} D_{N, s}
					\right|
			\\
			&\leq
				\frac{3}{\phi(s + 1)},
		\end{split}
	\end{equation}
  which implies $(\Delta W_{N,s})^2 \leq 9\phi(s + 1)^{-1}$. We have by Lemma~\ref{lemma:xi_slowly} that~$\xi^{-2}$ is a slowly varying function. By Karamata theory (see Theorem~\ref{thm:karamata} in the Appendix), the above inequality, and the fact that~$f$ is decreasing, we obtain 
	\begin{equation}
	\label{eq:quadvarlow1}
	\sum_{s=N}^t \left |\Delta W_{N,s} \right |^2 \le 16 \sum_{s=N}^t  s^{-2}\xi(s)^{-2}\le \frac{C}{N \xi(N)^2}=\frac{C}{\phi(N)\xi(N)}.
	\end{equation}
	We are going to stop our martingale when it becomes unexpectedly small. For this, let $\eta$ be the stopping time
	\begin{equation}
	\eta := \inf_{s \ge N} \left \lbrace W_{N,s} \le \frac{N}{\phi(N)}\right \rbrace.
	\end{equation} 
	By the Optional Stopping Theorem, it follows that $\{ W_{N,s\wedge \eta }\}_{s \ge N}$ is a martingale having the same expected value as $\{ W_{N,s}\}_{s \ge N}$. Furthermore, stopping a martingale can only decrease its increments. Thus, by Azuma's inequality, we obtain, by equation~\eqref{eq:quadvarlow1} and the definition of~$\xi$,
		\begin{equation}
		\mathbb{P}\left( W_{N,s\wedge \eta} \le \frac{2N}{\phi(N)} - \frac{N}{\phi(N)} \right) \le 
		 \exp \left\{   -C \frac{N^2}{\phi(N)^2} \phi(N)\xi(N)      \right\}\leq \exp \left \lbrace - CN \right \rbrace.
		\end{equation}
  Now, by the definition of~$\eta$, we obtain
  \begin{equation}
  \begin{split}
    \label{eq:freedmanms2}
    \P \left(     \exists s\in\N \text{ such that }  W_{N,s} \le \frac{N}{\phi(N)}        \right)
    &\leq 
    \P\left( \eta <\infty        \right) 
    \\
    &=\lim_{s\to\infty }\P \left( \eta \leq s      \right) 
    \\
    &=\lim_{s\to\infty } \P\left(    W_{N,s\wedge \eta}  \le \frac{N}{\phi(N)}                \right)
    \\
    &\leq
    \exp\left\{    -CN \right\}.
  \end{split}
\end{equation}
	By the pigeonhole Principle the above inequality implies~\eqref{eq:deglowerbound} since
	\begin{equation}
	W_{N,t} > \frac{N}{\phi(N)} \iff \sum_{i=1}^{N} d_t(i)\mathds{1}\{\tau(i)\le N \} > \frac{N\phi(t)}{\phi(N)},
	\end{equation} 
	implying the existence of at least one vertex with degree~$\phi(t)/\phi(N)$ among the~$N$ first vertices. 
\end{proof}

%% file: bootstrap.tex
In this section we prove Theorem \ref{thm:tau} which is a stronger version of Theorem \ref{thm:bootstrap}. The reader may find useful to recall the definition of the bootstrap percolation process at page \pageref{thm:bootstrap} as well as definition of $\tau_c$ in \eqref{def:tau}. For the sake of organization we will split the proof into several lemmas. 

Throughout this section, $f$ will be a function satisfying conditions $(\mathrm{D}_0)$, $\eqref{def:condS}$, and $\mathrm{RES}(-\gamma)$ for $\gamma \in [0,1)$. We will also use the notation $a_n \sim b_n$ to mean $a_n/b_n \to 1$ as $n\to \infty$. Moreover, whenever we use the letter $r$ when dealing with the graph process $\{G_t(f)\}_{t \ge 2}$, this $r$ is the exact same~$r$ as the threshold parameter of the Bootstrap Percolation process to be performed in the graphs $G_t(f)$. We then begin by some useful properties of $G_t(f)$ and $f$ when $f$ satisfies the aforementioned conditions. The first lemma is a small result concerning the speed of convergence to zero of $f$.
\begin{lemma}\label{l:ftlogt} Let $f$ be an edge-step function satisfying conditions $\eqref{def:condS}$ and $\mathrm{RES}(-\gamma)$ for $\gamma \in [0,1)$. Then,
	$$
		\lim_{t\to \infty} f(t)\log(t) = 0 
	$$
\end{lemma}
\begin{proof}In the regime~$\gamma>0$ the product~$f(t)\log(t)$ converges to zero as $t$ goes to infinity by virtue of Corollary~\ref{cor:a3}, since in this case by Theorem \ref{thm:repthm}, we have $f(t) = \ell(t)t^{-\gamma}$, where $\ell$ is a slowly varying function.  On the other hand, if~$f$ is slowly varying ($\gamma=0$), then for any~$K\in\N$, the monotonicity of~$f$ implies
	\begin{equation}\label{eq:condlog1}
	\begin{split}
	\sum_{s=K}^{Kt}\frac{f(s)}{s}\geq f(Kt)\sum_{s=K}^{Kt}\frac{1}{s} \geq  f(Kt)\log t (1-o(1))  \sim f(t)\log(t)(1-o(1)),
	\end{split}
	\end{equation}
	since $f$ is slowly varying. Finally, using condition $(S)$ which states that $\sum_sf(s)/s$ is finite and sending~$K$ to infinity, we obtain the desired result.
	
\end{proof}
It will be useful to our purposes to know the order of magnitude of the expected number of vertices in~$G_t(f)$. By the Karamata's Theorem (Theorem \ref{a:karamata}) we have that 
\begin{equation}\label{eq:EV}
	\E V_t(f) \sim \frac{f(t)t}{1-\gamma} \implies \E V_t(f) \le \frac{5}{4(1-\gamma)}f(t)t,
\end{equation}
for large enough $t$.

In order to prove Theorem \ref{thm:tau}, we will prove all the intermediate results and the theorem itself under the assumption that the sequence $(a_t)_{t \ge 1}$ satisfies the following technical conditions
\begin{equation}\label{def:condlog}\tag{L}
	\begin{split}
	(i)\quad  tf(t)f(a_t)\log (a_t) \to \infty; \; \quad (ii)\quad a_t \le t; \quad (iii) \quad f(a_t)\log (a_t) \geq f(t)\log t,
	\end{split}
\end{equation}
for sufficiently large~$t$. The reason why we need $(L)$ is of technical nature. Having such control over $(a_t)_{t \ge 1}$ makes some arguments easier to follow. However, the following lemma, whose proof is straightforward and will be omitted, guarantees that \eqref{def:condlog} is not restrictive.
\begin{lemma}[Monotonicity on $(a_t)_{t \ge 1}$] Let $(a_t)_{t \ge 1}$ and $(a'_t)_{t \ge 1}$ be two sequence of positive real numbers such that $a_t \le a'_t$ for all $t \ge 1$. Then, given a (multi)graph $G$, a positive constant $c$ and a natural number $r$, the following holds
	$$
	\mathbb{Q}_{G,a_t,r}\left(|\mathcal{I}_{\infty}| \ge c|V(G)|\right) \le \mathbb{Q}_{G,a'_t,r}\left(|\mathcal{I}_{\infty}| \ge c|V(G)|\right),
	$$
	for all $t \ge 1$.
\end{lemma}
Recall that the parameter $a_t$ is essentially the size of $\mathcal{I}_0$, which is the set of initially infected vertices. Thus, the above result states that if we increase the average size of $\mathcal{I}_0$, we increase our chances of percolating. The upshot is that we do not lose generality by imposing conditions on the growth rate of $(a_t)_{t \ge 1}$, as long as such conditions require making $(a_t)_{t \ge 1}$ go to infinity slower.

In general lines, in order to prove that $G_t(f)$ is supercritical for the Bootstrap percolation process, the key steps are to construct specific subgraphs in $G_t(f)$, which have good infectious properties in the sense that the infectious process reach those subgraphs more easily and they are interconnected in way that once the infection has reached them, it manages to spread to a positive proportion of the whole graph in a few steps.

In the next lemmas we will make sure $G_t(f)$ has some useful graph properties a.a.s.. We let $g$ and~$h$ be the following functions:
	\begin{equation}\label{def:gh}
		\begin{split}
			g(t) := \frac{1-\gamma}{f(t)}; \; h(t) := \frac{C_1}{6(r+1)^2C_2}\cdot\frac{g(t)}{\log(t)},
		\end{split}
	\end{equation}
where $r$ is the threshold parameter of the Bootstrap Percolation process and $C_1$ and $C_2$ are positive constants such that 
\begin{equation}\label{c1c2}
	C_1t \le \phi(t) \le C_2t.
	\end{equation}
Note that, by Lemma~\ref{l:ftlogt}, $h(t)$ goes to infinity as~$t$ increases.

We now present a sketch of the proof of the outbreak phenomenon, proving w.h.p. the existence of a structure that is highly susceptilble to infections. In what follows, everything happens w.h.p. as $t$ goes to infinity.

\begin{itemize}
	\item [(i)] At time~$t$ the graph has a number of vertices of order at most~$\approx f(t) t$, and a distinguished vertex~$v^*$ with degree of order at least~$\approx t/h(a_t)$;
	
	\item [(ii)] At time~$(r + 1) t$, the above bounds are still valid, but there exists a subgraph~$H$ whose \emph{total degree} in~$G_{(r + 1)t}(f)$ has order at least~$\approx t$, and every vertex of~$H$ sends at least~$r$ edges to~$v^*$;
	
	\item[(iii)] At time~$2(r + 1)^2 t$, the above is still valid, but now there exists a subset~$S \subset V(G_{2(r + 1)^2 t}(f))$ with cardinality of the same order as~$V(G_{2(r + 1)^2 t}(f))$ such that each vertex of~$S$ has at least~$r$ connections with~$H$. Other than that, the distinguished vertex~$v^*$ has order at least $\approx f(t) t/ h(a_t)$ distinct neighbours.
\end{itemize}

This is enough to prove our result for the graph~$G_{2(r + 1)^2 t}(f)$:

\begin{itemize}
	\item [(i)] The lower bound on the number of distinct neighbours of~$v^*$ and our choice of~$h$ imply that this vertex becomes infected in the first step;
	
	\item [(ii)] Since every vertex of~$H$ has at least~$r$ connections to~$v^*$, $H$ as a whole becomes infected in the second step;
	
	\item [(iii)] Now~$S$ has cardinality comparable to~$V(G_{2(r + 1)^2 t})$, and each vertex in~$S$ sends~$r$ connections to~$H$, implying that~$S$ becomes infected in the third step, implying the existence of an outbreak.
\end{itemize}

We can prove the same result for the graph~$G_t$ just by adjusting the constants. In accordance to the proof sketch above, for each time $t \in \N$, let $\mathcal{P}_t^1$ be the following collection of (multi)graphs
\begin{equation}\label{def:P1}
	\begin{split}
	\mathcal{P}^1_t := \left \lbrace G = (V,E) : |V| \le \frac{15f(t)t}{8(1-\gamma)} \text{ and } \exists v^* \in V, \text{with }d_Gv^*\ge \frac{C_1t}{C_2h(a_t)}\right \rbrace.
	\end{split}
	\end{equation}
	In words, $\mathcal{P}^1_t$ is the collection of multigraphs with at most $15f(t)t/8(1-\gamma)$ vertices and containing at least one vertex~$v^*$ whose degree is at least $C_1C_2^{-1}t/h(a_t)$, where $C_1$ and $C_2$ are the $f$-dependent constants in~\eqref{c1c2}. We call the distinguished vertex $v^*$ a \textit{star}. 
The next result ensures $G_t(f)$ has property $\mathcal{P}^1_t$ a.a.s..
\begin{lemma}\label{lemma:p1}Let $f$ be an edge-step function satisfying conditions $(S)$ and $\mathrm{RES}(-\gamma)$ for $\gamma \in [0,1)$. Then, there exists a $f$-dependent positive constant $C_f$ such that
	$$
	\P \left( G_t(f) \in \mathcal{P}^1_t \right) \ge 1- \exp\{-C_fh(a_t)\} - \exp\left \lbrace - \E V_t(f)/4 \right \rbrace.
	$$
	
\end{lemma}
\begin{proof}By Theorem~\ref{thm:deglowerbound}, choosing $N = h(a_t)$, we have that
	\begin{equation}\label{boundgrau}
	\P\left( \exists i \in \{1,2,\cdots, h(a_t)\}, \text{ such that }d_t(i) \ge \frac{C_1}{C_2}\frac{t}{h(a_t)}\right) \ge 1 - \exp\{-C_fh(a_t)\}
	\end{equation}
	for some $f$-dependent constant $C_f$. Since $V_t(f)$ is the sum of independent Bernoulli random variables and~$\E V_t(f) \le 5f(t)t/4(1-\gamma)$, a Chernoff bound yields
	\begin{equation}\label{eq:chernoffVt}
		\P\left(V_t(f) \ge \frac{15f(t)t}{8(1-\gamma)} \right) \le \exp\left \lbrace - \E V_t(f)/4 \right \rbrace,
	\end{equation}
	which proves that $G_t(f)$ has property $\mathcal{P}^1_t$ a.a.s.
\end{proof}
Now, if $H$ is a subset of vertices of a (multi)graph $G$, then $d_G(H)$, the degree of $H$, is simply the sum of the degree of the vertices in $H$. With this terminology in mind, let $\mathcal{P}_{(r + 1)t}^2$ be the property below
\begin{equation}\label{def:P2}
	\begin{split}
	\mathcal{P}_{(r + 1)t}^2 := \left \lbrace
	\begin{array}{c}
	 G = (V,E) : |V| \le \frac{15f((r + 1)t)(r + 1)t}{8(1-\gamma)}, \exists v^* \in V, \text{ s.t } d_G(v^*) \ge \frac{C_1t}{C_2h(a_{t})}, \\ \exists H \subset G \text{ s.t. } d_G(H) \ge t/8
	\\ \text{ and } \forall v \in H \text{ there are }r\text{ edges between }v\text{ and }v^*
	\end{array}
	\right \rbrace.
	\end{split}
\end{equation}
Now we will prove that $G_t(f)$ has property $\mathcal{P}_{(r + 1)t}^2$ a.a.s.
\begin{lemma}\label{lemma:p2}Let $f$ be an edge-step function satisfying conditions $(S)$ and $\mathrm{RES}(-\gamma)$ for $\gamma \in [0,1)$. Then, 
	$$
	\P \left( G_t(f) \in \mathcal{P}_{(r + 1)t}^2 \right) = 1- o(1).
	$$
\end{lemma}
\begin{proof} Since $f$ will be fixed the entire proof, to avoid clutter, we will drop the reference on $f$ in the notation~$G_t(f)$, writing simply $G_t$. We will prove first that conditioned on $G_t \in \mathcal{P}_t^1$, the graph $G_{(r+1)t}$ has property $\mathcal{P}_{(r+1)t}^2$ a.a.s. Thus, suppose that $G_t$ has property $1$ and let $v^*$ be its star. Let $j$ be a vertex whose degree is at least $g(t)$ at time~$t$. Notice that on the event $V_t(f) \le 15f(t)t/8(1-\gamma)$, there exists at least one such vertex, otherwise we would have
	 $$
	 \sum_{v \in G_t}d_t(v) \le \frac{15g(t)f(t)t}{8(1-\gamma)} = \frac{15t}{8} < 2t,
	 $$
	 which is a contradiction since $G_t$ is a graph with $t$ edges. Now, fix $k \in \{ 2, 3, \cdots, r+1\}$ and note that, for any $s \in [(k-1)t,kt]$, we have
	\begin{equation}
	\begin{split}
	\lefteqn{\mathbb{P}\left(j \text{ connects to } v^* \text{ at time }s+1 \; \middle | \; G_s, d_t(j)\ge g(t), G_t \in \mathcal{P}_t^1\right)}\phantom{************}\phantom{************}\phantom{*******}
	\\
	& = (1-f(s+1))\frac{d_s(j)d_sv^*}{2s^2} \\
	& \ge  \frac{C_1}{3(r+1)^2C_2}\frac{g(t)t}{t^2h(a_t)},
	\end{split}
	\end{equation}
	since the degree of a vertex is increasing in $s$, $s \le (r+1)t$ and $f\searrow 0$ implies that $1-f(t) \ge 2/3$ for sufficiently large $t$. Thus, using condition~\eqref{def:condlog}, our choice of $h$ leads to
	\begin{equation}
	\begin{split}
	\lefteqn{\mathbb{P}\left(j \text{ is not connected to } v^* \text{ in }G_{kt}\; \middle | \;G_{(k-1)t}, d_{(k-1)t}(j)\ge g(t), G_t \in \mathcal{P}_t^1\right)}\phantom{************}\phantom{*******************}
	\\
	& \le \left(1-\frac{C_1}{3(r+1)^2C_2}\frac{g(t)}{th(a_t)}\right)^t 
	\\
	&\le \exp\left\{   -2 \frac{f(a_t )\log(a_t)}{f(t)}                    \right\}
	\\
	&\le \exp\left\{   -2 \frac{f(a_t) \log(a_t)}{f(t)\log t}      \log t               \right\}
	\\
	&\le e^{-2\log t}.
	\end{split}
	\end{equation}
	Now, let $M_t(g(t))$ be the set of vertices in $G_t$ whose degree is at least $g(t)$. Observe that~$|M_t(g(t))| \le t$ almost surely. Then, for each $k \in \{2,\dots, r\}$ we have
	\begin{equation*}
		\P\left(\exists j \in M_t(g(t)) \text{ that does not connect to }v^* \text{ in the interval }[(k-1)t,kt] \; \middle | \; G_t \in \mathcal{P}_t^1\right) \le \frac{1}{t}.
	\end{equation*}
	Consequently, we have
	\begin{equation}
		\P\left(\exists j \in M_t(g(t))  \text{ that does not send } r \text{ edges to }v^* \text{ in } G_{(r+1)t} \; \middle | \; G_t \in \mathcal{P}_t^1 \right) \le \frac{r}{t}.
	\end{equation}
	Finally, let $H$ be the subgraph of $G_{t}$ composed by all its vertices whose degree is at least~$g(t)$. On the event where $V_t(f) \le 15f(t)t/8(1-\gamma)$ the sum of the degree of all vertices not in $H$ satisfies
	\begin{equation}
	\sum_{v \notin H}d_tv < \frac{15}{8(1-\gamma)}g(t)f(t)t = \frac{15t}{8}.
	\end{equation}
	Since the sum of all degrees equals in $G_t$ equals $2t$, by the above inequality, 
	\[
	d_t(H) :=\sum_{v\in H} d_t(v)\ge t/8 ,
	\]
	whenever $V_t(f) \le 15f(t)t/8$, which is satisfied when $G_t \in \mathcal{P}_t^1$. This shows that whenever $G_{t}$  is in~$\mathcal{P}^1_t $, then \textit{w.h.p} $G_{(r+1)t}$ has a subgraph~$H$ with total degree in $G_{(r+1)t}$ larger than~$t/8$, such that each vertex in $H$ shares at least $r$ edges with $v^*$. Moreover, when $v^*$ is seen as a vertex of~$G_{(r+1)t}$ it is a vertex of degree at least $C_1t/C_2h(a_t)$ as well, since the degree of a vertex is increasing in time. Therefore, using Lemma \ref{lemma:p1}, we obtain
	\begin{equation}\label{eq:p2}
		\begin{split}
			\mathbb{P}\left(G_{(r+1)t} \notin \mathcal{P}_{(r + 1)t}^2\right) & \le \mathbb{P}\left(G_{(r+1)t} \notin \mathcal{P}_{(r + 1)t}^2, G_{t} \in \mathcal{P}^1_{t}\right) + \mathbb{P}\left(G_{t} \notin \mathcal{P}^1_{t}\right) \\
			& \le \frac{r}{t}\P\left(G_t \in \mathcal{P}_t^1\right) + \P\left(G_{(r + 1)t}^1 \notin \mathcal{P}_{(r + 1)t}^1 \right)  + \mathbb{P}\left(G_{t} \notin \mathcal{P}^1_{t}\right) = o(1),
		\end{split}
	\end{equation}
	which completes the proof.
\end{proof}
We move to our third required property. To avoid clutter in its definition we define the properties we are interested in below. Fix a (multi)graph $G=(V,E)$ and a number $\zeta > 0$ and consider the graph properties:
\begin{enumerate}
\item $|V| \le \frac{15f((r + 1)^2 t)(r + 1)^2 t}{8(1-\gamma)}$;
\item there exists $v^* \in V$ whose degree is bounded from bellow by
	$$
	d_G v^* \ge \frac{C_1t}{C_2 h\left(a_{t}\right)}
	$$
\item there exists $H \subset V$ such that $d_G H \ge t/8$ and $\forall v \in H$ there are $r$ edges between $v$ and $v^*$;
\item there exists $S \subset V$ such that $|S|\ge \zeta f(t)t$ and each $v \in S$ sends at least $r$ edges to $H$.
\end{enumerate}

For a fixed number $\zeta > 0$, we let $\mathcal{P}_{(r + 1)^2 t}^3(\zeta)$ be the (multi)graph family defined below
\begin{equation*}\label{def:P3}
	\begin{split}
	\mathcal{P}_{(r + 1)^2 t}^3(\zeta) := \left \lbrace G=(V,E):G  \text{ satisfies all the properties (1)-(4) with parameter }\zeta\right \rbrace.
	\end{split}
\end{equation*}
\begin{lemma}\label{lemma:p3}Let $f$ be an edge-step function satisfying conditions $(S)$ and $\mathrm{RES}(-\gamma)$ for $\gamma \in [0,1)$. Then, there exists a number $\zeta_0 > 0$ depending on $r$ and $f$ only such that 
	$$
	\P \left( G_{(r + 1)^2 t}(f) \in \mathcal{P}_{(r + 1)^2 t}^3(\zeta_0)\right) = 1- o(1).
	$$
\end{lemma}
\begin{proof} Observe that $(1)$ is a consequence of Lemma \ref{lemma:p1} applied to $G_{(r+1)^2 t}(f)$. Whereas $(2)$ and $(3)$ are implied by Lemma \ref{lemma:p2} applied to $G_{(r + 1)^2 t}(f)$.

	Thus, we are left to prove that $(4)$ holds a.a.s. The way we prove this is by showing that conditioned on~$G_{(r + 1)t}(f) \in \mathcal{P}_{(r + 1)t}^2$, a positive proportion of vertices of degree $1$ in $G_{(r + 1)t}(f)$ connects to $H$ at time $2(r + 1)t$. Then we iterate this argument $r$ times, similarly to what we have done to show the existence of~$H$. For this, we need some definitions. Denote by $S_{(r + 1)t}(1)$ the random set consisting of vertices of degree exactly $1$ at time $(r + 1)t$ and let~$N_{(r + 1)t}(1)$ be its cardinality. Also, let $A_s$ and $C_s$ be the random sets 
	\begin{equation}
	\begin{split}
	A_s &:= \left \lbrace  v \in S_{(r + 1)t}(1) ; \; v \text{ does not connect to }H \text{ up to time }s \right \rbrace;
	\\
	C_s &:= \left \lbrace  v \in S_{(r + 1)t}(1) ;\; v \text{ connects to } H \text{ between times }(r + 1)t\text{ and }s \right \rbrace. 
	\end{split}
	\end{equation}
	We say a vertex in $A_s$ is \textit{available} to connect to~$H$. 
	At each step, we try to connect an available vertex to~$H$ using an edge-step. We are going to prove that \textit{w.h.p.} $|C_{2(r + 1)t}|$ is a positive proportion of~$S_{(r + 1)t}(1)$. So, let $\{Y_s\}_{s\ge t}$ be the following process
	\begin{equation}
	Y_s := \mathds{1}\left \lbrace \text{some vertex } v \text{ in } A_s\text{ connects to }H \text{ at time }s\right \rbrace.
	\end{equation}
	Note that $|C_{2(r + 1)t}| = \sum_{s=t}^{2(r + 1)t}Y_s$. Consider the following stopping time
	\begin{equation}
	\eta := \inf \left \lbrace s \ge t : |A_s| \le N_{(r + 1)t}(1)/2 \right \rbrace,
	\end{equation}
	To simplify our writing, we will denote by~$\mathbb{P}_{G_{(r + 1)t}}$ the law of~$\{G_s(f)\}_{s\geq (r + 1)t}$ conditioned on the event where graph~$G_{(r + 1)t}(f) = G_{(r + 1)t}$ for some given admissible~$G_{(r + 1)t} \in \mathcal{P}_{(r + 1)t}^2$. With this notation in mind, observe that, for~$s \in [(r + 1)t,2(r + 1)t]$ and large enough $t$
	\begin{equation}
	\mathbb{P}_{G_{(r + 1)t}}\left(Y_{s+1} = 1 \; \middle |\;G_{s}(f)\right) \ge  (1-f(s+1))\frac{|A_s|d_s H}{2s^2} \ge \frac{65|A_s|}{t},
	\end{equation}
	since $d_sH \ge t/8$ and $1-f(s+1) \ge 64/65$ for large $t$. Then, by the definition of $\eta$, it also follows that
	\begin{equation}\label{def:q}
	\mathbb{P}_{G_{(r + 1)t}}\left(Y_{s+1} = 1, \eta > s\; \middle | \; G_{s}(f)\right) \ge \frac{N_{(r + 1)t}(1)/130}{t}\mathds{1}_{\{\eta > s\}} =: q_{(r + 1)t}\mathds{1}_{\{\eta > s\}}.
	\end{equation}
	The above inequality implies that, on the event $\{\eta > 2(r + 1)t\}$, the sum $\sum_{s=(r + 1)t}^{2(r + 1)t}Y_s$ dominates a r.v. following a binomial distribution of parameters $(r + 1)t$ and $q_{(r + 1)t}$. Then,
	\begin{equation}
	\begin{split}
	\mathbb{P}_{G_{(r + 1)t}}\left( |C_{2(r + 1)t}| \le \frac{N_{(r + 1)t}(1)}{260}\right) & \le \mathbb{P}_{G_{(r + 1)t}}\left( \mathrm{bin}(t,q_{(r + 1)t}) \le  \frac{N_{(r + 1)t}(1)}{260}, \eta > 2(r + 1)t\right) 
	\\ &\quad +\mathbb{P}_{G_{(r + 1)t}}\left( |C_{2(r + 1)t}| \le  \frac{N_{(r + 1)t}(1)}{260}, \eta \le 2(r + 1)t \right) \\
	& =\mathbb{P}_{G_{(r + 1)t}}\left( \mathrm{bin}(t,q_{(r + 1)t}) \le  \frac{cN_{(r + 1)t}(1)}{4}, \eta > 2(r + 1)t\right),
	\end{split}
	\end{equation}
	since on $\{\eta \le 2(r + 1)t\}$ we have that $|C_{2(r + 1)t}| \ge N_{(r + 1)t}(1)/2$. Thus, by Chernoff bounds
	\begin{equation}
	\begin{split}
	\mathbb{P}_{G_{(r + 1)t}}\left( |C_{2(r + 1)t}| \le \frac{N_{(r + 1)t}(1)}{260}\right) & \le \mathbb{P}_{G_{(r + 1)t}}\left( \mathrm{bin}(t,q_{(r + 1)t}) \le \frac{N_{(r + 1)t}(1)}{260}\right)\le e^{- N_{(r + 1)t}(1)/520},
	\end{split}
	\end{equation}
	since $\mathrm{bin}(t,q_{(r + 1)t})$ has mean $N_{(r + 1)t}(1)/130$ (which is measurable with respect to $G_{(r + 1)t}$). Bow, a Chernoff bound and Theorem 1 of \cite{ARS17b} imply that, w.h.p.\ $N_{(r + 1)t}(1) \ge c_f f(t) t$ for an~$f$-dependent constant~$c_f > 0$ (indeed, just take~$A = o(\E V_{(r + 1)t}(f))$ in the Theorem mentioned, and use a Chernoff bound to show that $V_{(r + 1)t}(f)$ concentrates). The above upper bound combined with Lemma \ref{lemma:p2} gives us that $\P \left( C_{2(r + 1)t} \le (1/65) \cdot c_f f(t)t \right)$ is bounded from above by
	\begin{equation*}
	\begin{split}
	\lefteqn{\mathbb{P}\left( |C_{2(r + 1)t}| \le  N_{(r + 1)t}(1)/260,\, N_{(r + 1)t}(1) \ge c_f f(t)t,\, G_{(r + 1)t}(f) \in \mathcal{P}^{2}_{(r + 1)t} \right) }\quad
		\\ &\quad + \mathbb{P}\left(G_{(r + 1)t}(f) \notin \mathcal{P}_{(r + 1)t}^2\right)  + \mathbb{P}\left(N_{(r + 1)t}(1) \le ctf(t)\right) 
		\\ &= o(1),
	\end{split}
	\end{equation*}
	Now we can iterate this argument with the set $C_{2(r + 1)t}$ instead of $N_{(r + 1)t}(1)$ to guarantee \textit{w.h.p} that there exist at least $c'f(t)t$, for some constant $c'$ depending on $r$ and $f$, vertices sending at least two edges to $H$ in $G_{3(r + 1)t}(f)$ and so on. Choosing~$\zeta_0$ appropriately, and recalling that, by the fact $f$ is a regularly varying function, we have that
	$$
	\frac{f(t)t}{(r+1)f((r+1)t)t} \to (r+1)^{\gamma-1},
	$$
	as $t$ goes to infinity, is enough to conclude the proof.
	
\end{proof}
We define the last property we want our graphs to satisfy a.a.s. Given a vertex $v$ in a $G_t(f)$, we let $\Gamma_t(v)$ be the number of neighbors of $v$. Observe that $\Gamma_t(v) \le d_t(v)$. For two fixed positive numbers $\zeta$ and $\kappa$, consider the following list of graph properties
\begin{enumerate}
	\item[(1)'] $|V| \le \frac{15f(2(r + 1)^2t)2(r + 1)^2 t}{8(1-\gamma)}$;
	\item[(2)'] there exists $v^* \in V$ whose degree is bounded from bellow by
		$$
		d_G v^* \ge \frac{C_1t}{C_2 h\left(a_{t}\right)}; \quad \Gamma_t(v^*) \ge \kappa \frac{f(t)t}{h(a_t)};
		$$
	\item[(3)'] there exists $H \subset V$ such that $d_G H \ge t/16$ and $\forall v \in H$ there are $r$ edges between $v$ and $v^*$;
	\item[(4)'] there exists $S \subset V$ such that $|S|\ge \zeta f(t)t$ and each $v \in S$ sends at least $r$ edges to $H$.
	\end{enumerate}
	We then let $\mathcal{P}_{2(r + 1)^2 t}^4(\zeta, \kappa)$ be 
\begin{equation*}\label{def:P4}
	\begin{split}
	\mathcal{P}^4_{2(r + 1)^2 t}(\zeta, \kappa) := \left \lbrace G=(V,E): G \text{ satisfies all properties }(1)'-(4)' \text{ with parameters } \zeta\text{ and }\kappa \right \rbrace.
	\end{split}
\end{equation*}
Next result also guarantees that the above property holds a.a.s..
\begin{lemma}\label{lemma:p4}Let $f$ be an edge-step function satisfying conditions $(S)$ and $\mathrm{RES}(-\gamma)$ for $\gamma \in [0,1)$. Then, there exist positive parameters $\zeta_0$ and $\kappa_0$ depending on $f$ and $r$ only such that 
	$$
	\P \left( G_t(f) \in \mathcal{P}_{2(r + 1)^2 t}^4(\zeta, \kappa)\right) = 1- o(1).
	$$
\end{lemma}
\begin{proof} As usual we will prove that $G_{2(r + 1)^2 t}(f)$ satisfies $(1)'-(4)'$ for some oarameters $\zeta_0$ and $\kappa_0$. The property $(1)'$ is consequence of Lemma \ref{lemma:p1} applied directly to $G_{2(r + 1)^2 t}(f)$. Whereas $(3)'$ and $(4)'$ follow by conditioning on~$G_{(r + 1)^2 t}(f) \in \mathcal{P}_{(r + 1)^2 t}^3(\zeta)$ and seeing it as a subgraph of $G_{2(r + 1)^2 t}(f)$. 
	
	Thus we are left to prove $(2)'$.  The key step is to show $(2)'$ is seeing that, conditioned on $G_{(r + 1)^2 t}(f) \in \mathcal{P}_{(r + 1)^2 t}^3(\zeta)$, the graph~$G_{2(r + 1)^2 t}(f)$ has w.h.p a vertex $v^*$ such that $\Gamma_{2(r + 1)^2 t}(v^*) \ge\kappa f(t)t/h(a_t)$, for some positive number $\kappa$ w.h.p.. For this purpose, observe that if~$v^*$ has degree at least $C_1 C_2^{-1} t/h(a_t)$ at time $(r + 1)^2 t$, then for any~$s \in ((r + 1)^2 t,2(r + 1)^2 t)$, we have that
	\[
	\P\left(  \Gamma_{s+1}(v^*)- \Gamma_s(v^*) = 1 \; \middle |  \;\mathcal{F}_s, d_tv^* \ge \frac{C_1 t}{C_2h(a_{t})} \right) \ge \frac{C_1}{C_2} f(s)\frac{t}{2s h(a_{t})} \ge  \frac{C_1}{C_2} \frac{f(2(r + 1)^2 t )}{4(r + 1)^2 h(a_{t})}.
	\]
	Thus, using the fact that~$f$ is a regularly varying function with index~$-\gamma$, setting
	$$\kappa_1 = \frac{C_1}{4(r + 1)^{2 + \gamma} C_2} ,$$ 
	and conditioning on $G_{(r + 1)^2 t}(f) \in \mathcal{P}_{(r + 1)^2 t}^3(\zeta)$, the number of neighbors $v^*$ has at time $2(r + 1)^2 t$ dominates a binomial random variable with parameters~$(r + 1)^2 t$ and $\kappa_1 f(t)/2h(a_{t})$ for large~$t$. Recalling the definition of functions $h$ and $g$ at \eqref{def:gh} it follows that 
	$$
		\frac{\kappa_1 (r + 1)^2 t f(t)}{h(a_{t})} 
			= 
				\kappa_1 (r + 1)^2 \frac{C_1(1 - \gamma)}6 {C_2 (r + 1)^2}'tf(t)f(a_{t})\log(a_{t}) \to \infty,
	$$
	by condition \eqref{def:condlog}.
	Thus, conditioned on $G_{(r + 1)^2 t}(f) \in \mathcal{P}_{(r + 1)^2 t}^{3}$, by Chernoff bounds on the binomial $$\mathrm{bin}((r + 1)^2 t, \kappa_1 f(t)/h(a_{t})),$$ it follows that	
	the (multi)graph $G_{2(r + 1)^2 t}(f)$ has a vertex $v^*$ with at least $2 \kappa_1 t f(t)/3h(a_{t})$ neighbors. 
	Using that $G_{(r + 1)^2 t}(f) \in \mathcal{P}_{(r + 1)^2 t}^3(\zeta)$ a.a.s., and defining $kappa_0 = (2/3) \kappa_1$, we conclude the proof.
	
\end{proof}
Finally we are able to prove Theorem \ref{thm:tau} which implies Theorem \ref{thm:bootstrap} immediately.
\begin{proof}[Proof of Theorem \ref{thm:tau}] We first fix the parameter $r$ of the bootstrap percolation process  and define $\mathcal{G}_{t}$ as $\mathcal{P}_{t}^4(\zeta, \kappa)$, where $\zeta_0$ and $\kappa_0$ are the positive numbers given by Lemma~\ref{lemma:p4}. 
Now we are left to prove that the bootstrap percolation process on a (multi)graph in the collection $\mathcal{G}_{t}$ and parameters $a_t$ and $r$ \textit{w.h.p} has an outbreak of infection in at most $3$ steps.

	Summarizing all we have proven about $G_t(f)$ so far, a (multi)graph in $\mathcal{G}_t$ has the following structure:
	\begin{enumerate}
		\item $G_t(f)$ has at most $15f(t)t/8(1-\gamma)$ vertices;
		\item There exists a star $v^*\in G_t(f)$ such that
		\begin{equation}\label{eq:neigh}
				\Gamma_t(v^*) \ge \frac{\kappa_0 f(t)t}{h(a_{t/2(r+1)^2})};
		\end{equation}
		\item All vertices of degree at least $g(t)$ at time~$t /(2(r+1)^2)$ are connected $r$ times to $v^*$ in~$G_t(f)$. This is the subgraph $H$;
		\item There exists a subset $S\subset G_t$ containing at least $\zeta_0 f(t)t$ vertices whose elements are connected to $H$ at least $r$ times, that is, each vertex in $S$ sends at least $r$ edges to $H$.
	\end{enumerate}

	Having all the properties above in mind, we start the bootstrap percolation process. In the round zero we construct $\mathcal{I}_0$. We add each vertex in $G_t(f)$ to $\mathcal{I}_0$ independently and with probability $a_t/|V(G_t(f))|$. Observe that by $(1)$, it follows that 
	$$
	\frac{a_t}{|V(G_t(f))|} \ge \frac{c_1 a_t}{f(t)t},
	$$
	where $c_1 = 8(1-\gamma)/15$. 
	
	Now, observe that the number of neighbors of $v^*$ infected at round zero dominates a binomial random variable with parameters $\kappa_0 f(t)t/h(a_{t/2(r+1)^2})$ and $c_1 a_t/f(t)t$. Since $f$ is a regularly varying function with index $-\gamma \in (-1,0]$, and $a_t$ is increasing to infinity, Corollary~\ref{cor:a3} implies that $a_t/h(a_{t/2(r+1)^2})$ goes to infinity as $t$ goes to infinity. This and Chernoff bounds implies that 
	\begin{equation}
		\mathbb{Q}_{G_t(f),a_t,r}\left( \mathcal{I}_0 \text{ contains at least }r \text{ neighbors of }v^*\right) = 1-o(1).		
	\end{equation}
	In words, at round zero $v^*$ has at least $r$ infected neighbors \textit{w.h.p}. Now observe that given that $v^*$ has at least $r$ neighbors in $\mathcal{I}_0$, it follows that at round one $v^*$ gets infected. Then, at the next round all the subgraph $H$ gets infected since all of its vertices send at least $r$ edges to $v^*$. Finally at round  three, $S$ gets infected since all of its vertices send at least $r$ edges to $H$. This proves that $\mathcal{I}_3$ is a positive proportion of~$G_t(f)$ and proves our theorem.
\end{proof}

%% file: appendix.tex
\section{Martingale concentration inequalities}\label{a:martingconc}

For the sake of completeness we state here two useful concentration inequalities for martingales which are used throughout the paper.
\begin{theorem}[Azuma-H\"{o}ffeding Inequality - see \cite{CLBook}]\label{thm:azuma} Let $(M_n,\mathcal{F})_{n \ge 1}$ be a martingale satisfying
	\[ 
	\lvert M_{i+1} - M_i \rvert \le a_i
	\]
	Then, for all $\l > 0 $ we have
	\[
	\P \left( | M_n - M_0 | > \lambda \right) \le \exp\left( -\frac{\lambda^2}{\sum_{i=1}^n a_i^2} \right).
	\]
\end{theorem}

\section{Karamata theory}\label{a:karamata}

The three following results are used throughout the paper.
\begin{corollary}[Representation theorem - Theorem~$1.4.1$ of~\cite{regvarbook}]\label{thm:repthm} Let $f$ be a continuous regularly varying function with index of regular variation $\gamma$. Then, there exists a slowly varying function $\ell$ such that
	\begin{equation}
	f(t) = t^{\gamma}\ell(t),
	\end{equation}
	for all $t$ in the domain of $f$.
\end{corollary}
\begin{corollary}\label{cor:a3}Let $f$ be a continuous regularly varying function with index of regular variation $\gamma <0$. Then,
	\begin{equation}
	f(x) \to 0,
	\end{equation}
	as $x$ tends to infinity. Moreover, if $\ell$ is a slowly varying function, then for every $\varepsilon >0$
	\begin{equation}
	x^{-\varepsilon} \ell(x) \to 0 \text{ and } x^{\varepsilon}\ell(x) \to \infty
	\end{equation}
	\begin{proof}
		Comes as a straightforward application of Theorem~$1.3.1$ of~\cite{regvarbook} and~Corollary~\ref{thm:repthm}.
	\end{proof}
\end{corollary}
\begin{theorem}[Karamata's theorem - Proposition~$1.5.8$ of~\cite{regvarbook}]\label{thm:karamata} Let $\ell$ be a continuous slowly varying function and locally bounded in $[x_0, \infty)$ for some $x_0 \ge 0$. Then
	\begin{itemize}
		\item[(a)] for $\alpha > -1$
		\begin{equation}
		\int_{x_0}^{x}t^{\alpha}\ell(t)dt \sim \frac{x^{1+\alpha}\ell(x)}{1+\alpha}. 
		\end{equation}
		
		\item[(b)] for $\alpha < -1$
		\begin{equation}
		\int_{x}^{\infty}t^{\alpha}\ell(t)dt \sim \frac{x^{1+\alpha}\ell(x)}{1+\alpha}.
		\end{equation}
	\end{itemize}
\end{theorem}